\documentclass[11pt]{article}

\usepackage[utf8]{inputenc}
\usepackage{graphicx, amsmath, amsthm, amsfonts, amssymb, subfigure, amstext}
\usepackage{wrapfig, amsopn, latexsym,  epsfig, multicol, multirow,verbatim}
\usepackage{color, comment,  caption,  bm}
\usepackage{mathbbol}
\usepackage[numbers]{natbib}
\usepackage{geometry}
\usepackage{cmbright}

\newtheorem{Theorem}{Theorem}[section]
\newtheorem{Remark}{Remark}[section]
\newtheorem{Definition}{Definition}[section]

\newtheorem{prop}{Proposition}[section]
\newcommand{\comments}[1]{}

\title{Flexible models for overdispersed and underdispersed count data}
\author{Dexter Cahoy\thanks{Department  of Mathematics and Statistics,  University of Houston-Downtown, USA, email: {\tt cahoyd@uhd.edu}} \and Elvira Di Nardo\thanks{Department of Mathematics ``G.~Peano'',  University  of Torino,  Italy, emails: \texttt{elvira.dinardo@unito.it, federico.polito@unito.it}} \and Federico Polito\footnotemark[2]}

\date{}

\begin{document}

\maketitle

\begin{abstract}
\indent

Within the framework of probability models for overdispersed count data, we propose the generalized fractional Poisson distribution (gfPd), which is a natural generalization of the fractional Poisson distribution (fPd), and the  standard Poisson distribution. We derive some properties of gfPd and more specifically we study moments, limiting behavior and other features of fPd. The skewness  suggests that  fPd  can be  left-skewed, right-skewed or symmetric; this makes the model flexible  and appealing in practice. We apply the model to real big count data and estimate the model parameters using maximum likelihood.
Then, we turn to  the very general class of weighted Poisson distributions (WPD's) to allow both overdispersion and underdispersion. Similarly to Kemp's generalized hypergeometric probability distribution, which is based on hypergeometric functions, we analyze a class of WPD's related to a generalization of Mittag--Leffler functions. The proposed class of distributions includes  the well-known COM-Poisson and the hyper-Poisson models. We characterize conditions on the parameters allowing for overdispersion and underdispersion,  and analyze two  special cases of interest which have not yet appeared in the literature.

\vspace{0.1in}

\end{abstract}

\section{Introduction and mathematical background}

The negative binomial distribution is one of the most widely used  discrete probability models that  allow departure from the mean-equal-variance Poisson model. More specifically, the negative binomial distribution models overdispersion of data relative to the Poisson
distribution.  For clarity, we refer to  the extended negative binomial distribution with  probability mass function  
\begin{equation}
P(X=x) = \frac{\Gamma (r + x) }{\Gamma (r) x!}  p^r(1-p)^x, \qquad x=0, 1, 2, \ldots,
\end{equation}
where $r>0$. If $r \in \{1,2,\dots\}$,  $x$ is the number of  failures which occur in a sequence of independent Bernoulli trials to obtain $r$ successes, and $p$ is the success probability of  each trial.

One limitation of the negative binomial distribution in fitting overdispersed count data is that the skewness and kurtosis are always positive. An example is given in Section \ref{simpar}, in which we introduce two real world data sets that do not fit a negative binomial  model.  The data sets reflect reported incidents of crime that occurred in the city of Chicago from  January 1, 2001 to May 21, 2018. These data sets are overdispersed but the skewness coefficients are estimated to be respectively -0.758 and -0.996. Undoubtedly, the negative binomial model is expected to underperform in these types of count populations. 
These data sets are just two examples in a yet to be discovered  non-negative binomial world, thus demonstrating the real need for a more flexible alternative for overdispersed count data. The literature on alternative probabilistic models for overdispersed count data is vast. A history of the overdispersed  data problem and related literature can be found in \cite{smkbp05}. In this paper  we consider the fractional Poisson distribution (fPd) as an  alternative. The fPd arises naturally from  the widely studied fractional Poisson process \citep{saz97, ras00,gju01, las03,  bno09,  cuw10, mnv11}.  It has not yet been studied in depth and  has not been applied to model real count data. We show that the fPd allows big (large mean), both left- and right-skewed overdispersed count data making it attractive for practical settings, especially now that data are becoming more available and bigger than before. fPd's  usually involve one parameter; generalizations to two parameters are proposed in \cite{bno09,Herrmann}. Here, we take a step forward and further generalize the fPd to a three parameter model, proving the resulting distribution is still overdispersed.

One of the most popular measures to detect the departures from the Poisson distribution is the so-called Fisher index which is the ratio of the variance to the mean $(\lessgtr 1)$ of the count distribution. As shown in the crime example of Section \ref{simpar}, the computation of the Fisher index is not sufficient to determine a first fitting assessment of the model, which indeed should take into account at least the presence of negative/positive skewness. To compute all these measures, the first three factorial moments should be considered.  Consider a discrete random variable $X$ with probability generating function (pgf)
\begin{equation}
\label{forma}
 G_X(u) =   \bm{E} u^X  = \sum_{k\geq 0} a_k \frac{(u-1)^k}{k!}, \qquad |u|\le 1, 
\end{equation}
where $\{a_k\}$ is a sequence of real numbers such that $a_0=1.$
Observe that $Q(t)=G_X(1+t)$ is the factorial moment generating function of $X.$ The $k$-th moment  is
\begin{equation}
\bm{E} X^k = \sum_{r=1}^k S(k,r) a_r,
\label{momfac} 
\end{equation}
where $S(k,r)$ are the Stirling numbers of the second kind \cite{DiNardo}. By means of the factorial moments it is straightforward to characterize overdispersion or underdispersion as follows:  letting 
$a_2 > a_1^2$ yields overdispersion whereas $a_2 < a_1^2$ gives underdispersion.
Let $c_2$ and $c_3$ be the second and third cumulant of $X$, respectively. Then,
the skewness can be expressed as
\begin{align}
    \gamma(X) = \frac{c_3}{c_2^{3/2}} =
    \frac{a_3+3a_2+a_1[1-3 a_2 + a_1(2a_1-3)]}{(a_1+a_2-a_1^2)^{3/2}}.
    \label{skew}
\end{align}
If the condition
\begin{equation}
\lim_{n \rightarrow \infty}
\frac{a_n}{(n-k)!} = 0, \qquad k \leq n,
\label{condition}
\end{equation}
is fulfilled, the  probability  mass  function   of $X$ can be written in terms of its factorial moments \cite{daley}:
\begin{equation}
P(X=x) = \frac{1}{x!} \sum_{k\geq 0} a_{k+x} \frac{(-1)^k}{k!}, \qquad x \geq 0.
\label{mf1}
\end{equation}

As an example, the very well-known generalized Poisson distribution which accounts for both under and overdispersion \cite{macedo, consul}, put in the above form has factorial moments given by $a_0=1$ and
\begin{align}
    a_k = \sum_{r = 0}^{h(\lambda_2)} \frac{1}{r!} \lambda_1 (\lambda_1 +\lambda_2(r+k))^{r+k-1} e^{-(\lambda_1+\lambda_2(r+k))}, \qquad \lambda_1 > 0,
\end{align}
where $h(\lambda_2) = \infty$ and $k =1,2,\ldots$, if $\lambda_2>0$. While $h(\lambda_2) = M-k$ and $k=1, \ldots, M$, 
if $\max(-1, -\lambda_1/M) \le \lambda_2 < 0$ and $M$ is the largest positive integer for which $\lambda_1+M\lambda_2 > 0$.

Another example is given by the Kemp family of generalized hypergeometric factorial moments distributions (GHFD) \cite{kak74} for which the factorial moments are given by
\begin{equation}
a_k = \frac{\Gamma \left[ (a+k);(b+k) \right] \lambda^k }{ \Gamma \left[  (a); (b) \right]}, \qquad k \geq 0, \label{B1}
\end{equation}
where $ \Gamma \left[  (a); (b) \right] =\prod_{i=1}^p \Gamma (a_i)/\prod_{j=1}^q \Gamma (b_j)$,
with $a_1, \ldots, a_p, b_1, \ldots, b_q \in {\mathbb R}$ and $p,q$ non negative integers. The factorial moment generating function is $Q(t) = \, _pF_q \left[ (a) ;(b); \lambda t \right]$, where
\begin{equation}
_pF_q \left[ (a) ;(b); z \right]=  \, _pF_q(a_1,\ldots, a_p; b_1,\ldots,b_q; z) = \sum_{m\geq 0} \frac{(a_1)_m \cdots (a_p)_m}{(b_1)_m \cdots (b_q)_m} \frac{z^m}{m!},
\label{hyper}
\end{equation}
and $(a)_m=a(a+1)\cdots(a+m-1), m \geq 1.$ Both overdispersion and underdispersion are possible, 
depending on the values of the parameters \cite{trip1979}.
The generalized fractional Poisson distribution (gfPd), which we introduce in the next section, lies in the same class of the Kemp's GHFD but with the hypergeometric function in \eqref{hyper} substituted by a generalized Mittag--Leffler function (also known as three-parameter Mittag--Leffler function or Prabhakar function). In this case, as we have anticipated above, the model is capable of not only describing overdispersion but also having a degree of flexibility in dealing with skewness.

It is worthy to note that there exists a second family of Kemp's distributions, still based on hypergeometric functions and still allowing both underdispersion and overdispersion. This is known the Kemp's generalized hypergeometric probability distribution (GHPD) \cite{kemp1968} and it is actually a special case of the very general class of weighted Poisson distributions (WPD).
Taking into account the above features, we thus analyze the whole class of WPD's with respect to the possibility of obtaining under and overdispersion.
In Theorem \ref{rop} we first give a general necessary and sufficient condition to have an underdispersed or an overdispersed WPD random variable in the case in which the weight function may depend on the underlying Poisson parameter $\lambda$. Special cases of WPD's admitting a small number of parameters have already proven to be of practical interest, such as for instance the well-known COM-Poisson \cite{COMPOISS} or the hyper-Poisson \cite{CB1964} models.
Here we present a novel WPD family related to a generalization of Mittag-Leffler functions in which
the weight function is based on 
a ratio of gamma functions. The proposed distribution family includes the above-mentioned well-known classical cases. We characterize conditions on the parameters allowing overdispersion and underdispersion  and analyze two further special cases
of interest which have not yet appeared in the literature. We derive recursions to generate probability mass functions (and thus random numbers) and show how to approximate the mean and the variance.

The paper is organized as follows: in Section \ref{asez}, we introduce the generalized fractional Poisson distribution, discuss some properties and recover the classical fPd as a special case. These models are fit to  the two real-world data sets mentioned above. Section \ref{wpdd} is devoted to weighted Poisson distributions, their characteristic  factorial moments and the related conditions to obtain overdispersion and underdispersion. Furthermore, the novel WPD based on a generalization of Mittag--Leffler functions  is introduced and described in Section \ref{further}:  we discuss some properties and show how to get exact formulae for factorial moments by using Fa\`a di Bruno's formula \cite{Stanley}. Two special models are then characterized depending on the values of the parameters and compared to classical models. Finally, some illustrative plots end the paper.

\section{Generalized fractional Poisson distribution (gfPd)}\label{asez}

\begin{Definition}
A random variable  $X_{\alpha, \beta}^\delta \; \stackrel{d}{=} \; $gfPd$(\alpha, \beta, \delta, \mu)$ if
\begin{equation}
P(X_{\alpha, \beta}^\delta = x) =  \frac{   \Gamma (\delta + x)}{x! \Gamma (\delta)} \mu^x \Gamma (\beta) E_{\alpha,\alpha x +\beta}^{\delta  +x} (-\mu), \quad \mu >0;\, x \in \mathbb{N}; \, \alpha, \beta \in (0, \;1]; \, \delta \in  (0,  \;  \beta/\alpha],
\label{ciccia}
\end{equation}
where
\begin{equation}
\label{igo}
	E_{\eta,\nu}^\tau( w) = \sum_{j=0}^\infty \frac{(\tau)_j}{j!\Gamma(\eta j+\nu)} \;  w^j,
\end{equation}
$w  \in \mathbb{C};  \Re(\eta),\Re(\nu),\Re(\tau) >0$, is the generalized Mittag--Leffler function \citep{pra71} and  $(\tau)_j = \Gamma ( \tau + j  )/  \Gamma ( \tau )  $ denotes the Pochhammer symbol. 
\end{Definition}
To show non-negativity, notice that 
\begin{equation}
  \frac{  \Gamma (\delta + x)}{ \Gamma (\delta)}   E_{\alpha,\alpha x +\beta}^{\delta  + x} (-\mu) \geq 0 \iff    (-1)^x \frac{ d^x}{d \mu^x}   E_{\alpha, \beta}^{\delta} (-\mu) \geq 0,
\end{equation}
that is, $ E_{\alpha, \beta}^{\delta} (-\mu)$ is completely monotone. From \cite{deol}, it is known that  $ E_{\alpha, \beta}^{\delta} (-\mu)$ is completely monotone if  $\alpha, \beta \in (0, \;1], \delta \in  (0,  \;  \beta/\alpha]$ and thus the pmf in (\ref{ciccia}) is non-negative.

Note that the probability mass function can be determined using the following integral representation \citep{poltom}:
\begin{equation}
   P(X_{\alpha, \beta}^\delta = x) = \frac{ \Gamma(\beta) }{x! \Gamma(\delta)} \mu^x \int_{\mathbb{R}^+}e^{- \mu    y} y^{\delta+x -1}  \phi(-{\alpha},
   \beta- \alpha \delta; -y) dy,  \label{form1bis}
\end{equation}
where  the Wright function $\phi$  is defined as the convergent sum \citep{kilbas} 
\begin{equation}
    \phi(\xi,\omega; z)  =\sum\limits_{r=0}^\infty \frac{ z^r}{r! \Gamma [\xi r + \omega ]}, \qquad \xi > -1, \omega,z \in {\mathbb R}.
\end{equation}
\begin{Remark}\label{momf}
The random variable $X_{\alpha,\beta}^\delta$ has factorial moments
\begin{equation}
a_{k} =  \frac{\Gamma(\beta) \Gamma(\delta + k)}{\Gamma (\alpha k + \beta) \Gamma(\delta)} \mu^{k}, \qquad k \geq 0.
\label{factgfPd}
\end{equation}
Hence the pgf is $G_{X_{\alpha, \beta}^\delta}(u)= \Gamma (\beta)  E_{\alpha, \beta}^{\delta} ( \mu (u-1) )$, $|u| \le 1$.
\end{Remark}

By expressing the moments in terms of factorial moments and after some algebra we obtain 
\begin{align}
\bm{E} [X_{\alpha, \beta}^\delta]  & = \frac{ \Gamma ( \beta ) \delta  \mu}{ \Gamma ( \beta + \alpha)}, \\
\bm{V}\text{\textbf{ar}}[ X_{\alpha, \beta}^\delta]  &=  \frac{ \Gamma ( \beta ) \delta  \mu}{ \Gamma ( \beta + \alpha)} + \Gamma ( \beta ) \delta \mu^2 \left( \frac{  (\delta +1) }{ \Gamma ( \beta + 2\alpha)}  -  \frac{ \Gamma ( \beta ) \delta }{ \Gamma ( \beta + \alpha)^2}  \right).
\end{align}
\begin{Theorem}
$X_{\alpha, \beta}^\delta$ exhibits overdispersion. 
\end{Theorem}
\begin{proof}
We have 
\begin{align}
a_2 > a_1^2 \Leftrightarrow  \frac{\delta +1}{\Gamma (2 \alpha +\beta)} > \frac{\delta \Gamma (\beta)}{\Gamma^2 ( \alpha +\beta)}  
\Leftrightarrow \delta \bigg( \frac{\Gamma(\beta)}{\Gamma^2(\alpha+\beta)} - \frac{1}{\Gamma(2 \alpha+\beta)} \bigg) < \frac{1}{\Gamma(2 \alpha + \beta)}
\end{align}
and
\begin{align}
\frac{\Gamma(\beta)}{\Gamma^2(\alpha+\beta)} - \frac{1}{\Gamma(2 \alpha+\beta)} > 0 \quad {\rm as}
\quad \text{Beta}(  \beta, \alpha) > \text{Beta}(\alpha+\beta, \alpha).
\end{align}

Thus, the distribution is overdispersed for
\begin{equation}
\delta < \frac{\text{Beta}(  \alpha + \beta, \alpha)}{\text{Beta}(  \beta, \alpha) - \text{Beta}(  \alpha + \beta, \alpha)}.
\label{(bound)}
\end{equation}
Observe that the function $\beta \text{Beta}(\beta,\alpha)$ is increasing in $\beta$ for $\alpha, \beta \in (0,1)$ 
as 
\begin{equation}
\frac{\partial}{\partial \beta} \beta \text{Beta}(\beta,\alpha) = \text{Beta}(\beta,\alpha) (1 + \beta (\psi(\beta)-\psi(\alpha+\beta))>0,
\label{(deriv)}
\end{equation}
where $\psi$ is the digamma function. Note that \eqref{(deriv)} is positive by formula (1.3.3) of \cite{leb72} as $\psi$ is increasing on $(0,\infty)$. 
Thus
\begin{equation}
\beta \text{Beta}(  \beta, \alpha) <  (\alpha + \beta)  \text{Beta}(  \alpha + \beta, \alpha)\Leftrightarrow  \frac{\beta}{\alpha} < \frac{\text{Beta}(  \alpha + \beta, \alpha)}{\text{Beta}(  \beta, \alpha) - \text{Beta}(  \alpha + \beta, \alpha)}
\label{(bound1)}
\end{equation}
and for $\delta \in (0, \beta/\alpha)$    the  bound \eqref{(bound)} is always verified. 
\end{proof}

\subsection{Fractional Poisson distribution}

This section analyzes the classical fPd, which is a special case of gfPd, and is obtained when $\beta=\delta=1$. The fPd can model  asymmetric (both left-skewed and right-skewed) overdispersed count data for all mean  count values (small and large).    The fPd  has probability mass function (pmf)    
\begin{equation}
P(X_\alpha=x) =  \mu^x E_{\alpha,\alpha x +1}^{x +1} ( -\mu) , \qquad x=0,1,2,\ldots, 
\label{mf2}
\end{equation}
where $\mu >0$, $\alpha \in [0,1]$.

Notice that if  $\alpha = 1$, the standard  Poisson  distribution is retrieved, while for $\alpha=0$ we have 
$X_{0} \stackrel{d}{=} \text{Geo} \left( 1/(1+ \mu)\right)$. Indeed,
\begin{align}
P(X_0=x) =  \frac{\mu^x}{x!}\sum_{j=0}^{\infty}\frac{(j+x)!}{j!}(-\mu)^j = \frac{1}{1+\mu}\left(\frac{\mu}{1+\mu}\right)^x, \qquad x \ge 0.
\end{align}

Furthermore, the probability mass function can be determined using the following integral representation \citep{bao10}:
\begin{equation}
    P(X_\alpha=x) = \frac{ \mu^x }{x!}\int_{\mathbb{R}^+}e^{- \mu    y} y^x  M_{\alpha} (y) dy,  \label{form1}  
\end{equation}
where  the $M$-Wright  function \citep{mmp10} 
\begin{equation}
    M_\alpha (y) =\sum\limits_{j=0}^\infty \frac{ (-y)^j}{j! \Gamma [-\alpha j + (1-\alpha) ]}  =  \frac{1}{\pi} \sum\limits_{j=1}^\infty \frac{ (-y)^{j-1}}{(j-1)!} \Gamma (\alpha j) \sin ( \pi \alpha j)
\end{equation}
is the probability density function of the
random variable $S^{-\alpha}$  with $S \stackrel{d}{=} \alpha^+$-stable supported in $\mathbb{R}^+.$ By using (\ref{form1}), 
the  cumulative distribution function turns out to be
\begin{equation}
    F_{X_\alpha} (x)= \sum_{r=0}^\infty  \binom{x + r-1}{x} \frac{(-1)^r \mu^{-(r+1)}}{\Gamma (1 - \alpha (r+1))} {\mathbb 1}_{(x > 0)}(x).
\end{equation}

\begin{Remark}
From \eqref{momf}, the random variable $X_\alpha$ has factorial moments
\begin{equation}
a_{k} = \frac{\mu^{k} k!}{\Gamma (1 
+ \alpha k)}, \qquad k \geq 0.
\label{factfPd}
\end{equation}
Hence the probability generating function is $G_{X_\alpha}(u)= E_{\alpha,1}^1 \left( \mu \left( u-1\right)\right)$, $|u| \le 1$.

\end{Remark}

With respect to the symmetry structure of $X_\alpha$, from \eqref{skew} and \eqref{factfPd}, the skewness of $X_\alpha$ reads 
\begin{equation}
\gamma(X_{\alpha})= \frac{\frac{1}{\mu^2 \Gamma (1 + \alpha) } + \frac{6}{\mu \Gamma (1 + 2\alpha) } + \frac{6}{\Gamma (1 + 3 \alpha) } -\frac{3}{\mu \left[\Gamma (1 + \alpha)\right]^2 } -\frac{6}{\Gamma (1 + \alpha) \Gamma (1 + 2 \alpha) }  + \frac{2}{[\Gamma (1 + \alpha) ]^3}}{ \left(\frac{1}{\mu \Gamma (1 + \alpha) } + \frac{2}{\Gamma (1 + 2\alpha) } - \frac{1}{ [ \Gamma (1 + \alpha)]^2} \right)^{3/2}  }.  
\end{equation}
Moreover,
\begin{equation}
\lim_{\mu \to \infty}  \gamma(X_{\alpha}) = \frac{  \frac{6}{\Gamma (1 + 3 \alpha) } - \frac{6}{\Gamma (1 + \alpha) \Gamma (1 + 2 \alpha) }  + \frac{2}{\left[\Gamma (1 + \alpha) \right]^3}}{ \left(  \frac{2}{\Gamma (1 + 2\alpha) } - \frac{1}{ [ \Gamma (1 + \alpha)]^2} \right)^{3/2}}  \neq 0, 
\label{skew1}
\end{equation}
which correctly vanishes if $\alpha=1$, like the ordinary Poisson distribution.

\subsubsection{Simulation and parameter estimation}\label{simpar}

The  integral representation  (\ref{form1}) allows  visualization of the probability mass function of $X_\alpha$ (see Figure \ref{Fig1}).
Figure \ref{Fig1}  shows the flexibility of the fPd. The probability distribution ranges from zero-inflated  right-skewed ($\alpha \to 0$) to left-skewed ($\alpha \to 1$)  and  symmetric  ($\alpha = 1$) overdispersed count data.     
\begin{figure}[h!t!b!p!]
\centering
\includegraphics[scale=.45]{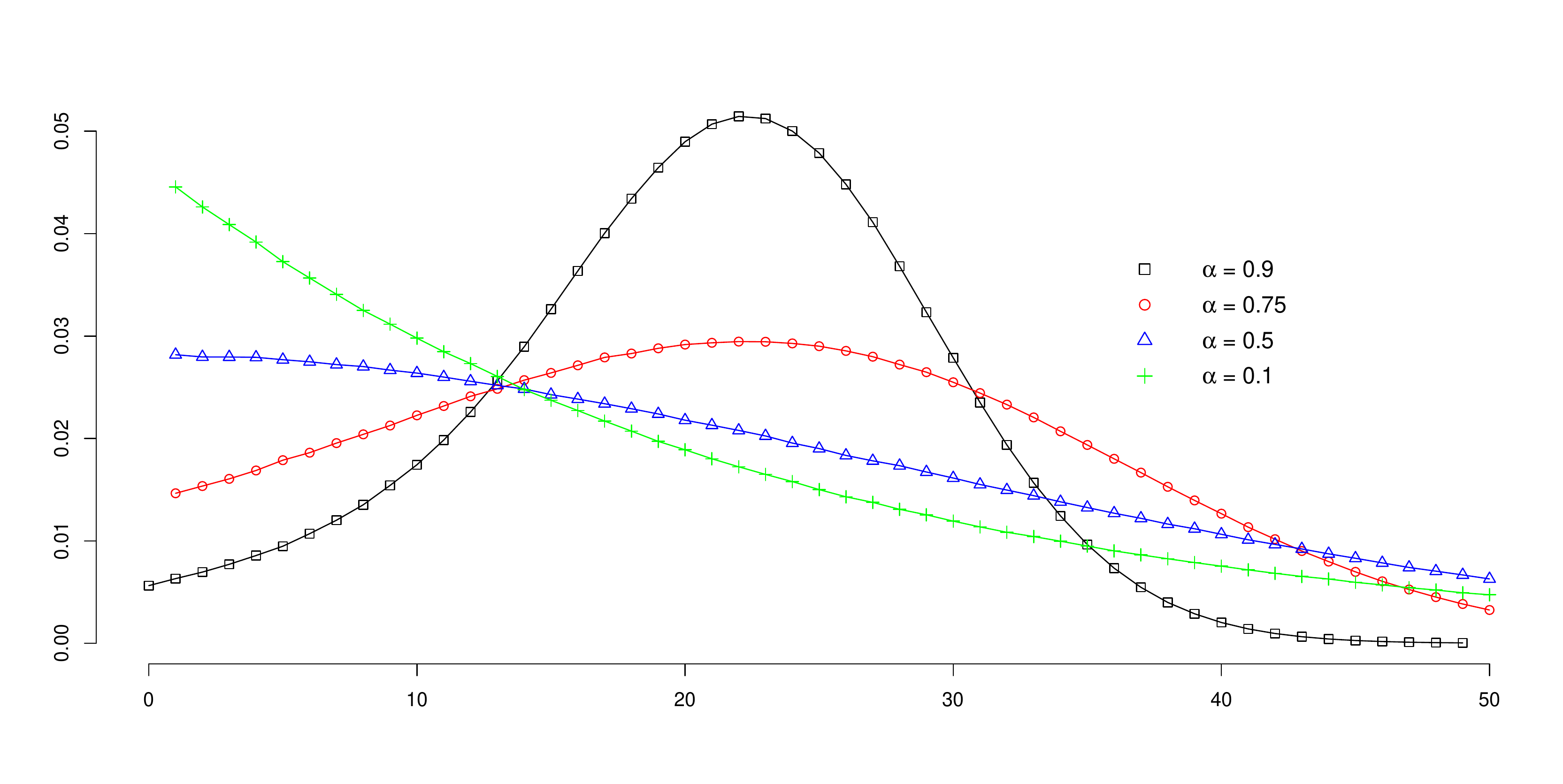}
\caption{Probability mass functions of $X_\alpha$ for $\alpha =0.1, 0.5, 0.75, 0.9$, and $\mu = 20$.}
\label{Fig1}
\end{figure}
To compute the integral in (\ref{form1}) by means of Monte Carlo techniques,  we use the approximation, 
\begin{equation}
p_x^{\alpha}  \approx  \frac{\mu^x}{\; x!} 
\left(\frac{1}{N} \sum_{j=1}^N e^{- \mu Y_j} Y_j^x \right),
\end{equation}
where   $Y_j's \; \stackrel{iid}{=}   S^{-\alpha}.$  Note that  the random variable $S$ can be generated using the following formula \citep{kan75, cms76}:
\begin{equation}
S \stackrel{d}{=}  \frac{\sin(\alpha \pi U_1)[ \sin((1-\alpha)\pi
U_1)]^{1/\alpha-1}}{[\sin (\pi U_1)]^{1/\alpha}|\ln
U_2|^{1/\alpha-1}}, 
\end{equation}
where $U_1$  and $U_2$  are independently and uniformly distributed in $[0,1]$.  Thus, fractional Poisson random numbers can be generated using the algorithm below.

\begin{quote}
\textbf{Algorithm:}

\textbf{Step 1.} Set $X=0,$ and $T=0.$

\textbf{Step 2.}  While $\lbrace T\leq 1 \rbrace$
\begin{eqnarray*}  
T &=& T+  V^{1/\alpha} \, S   \\
X &=& \text{ifelse}(T \leq1,  X+1, X) 
\end{eqnarray*}

\textbf{Step 3.} Repeat steps $1-2, n$ times. 
\end{quote}
 
Note that the random variable $V$ follows the exponential distribution with density function $\mu \exp (-\mu v), v \geq 0.$ Algorithms for generating random variables from the exponential density function are well-known.    Hence, the algorithm allows estimation of the $k$th moment, i.e.,  $\bm{E} X_\alpha^{k}.$

Figure \ref{Fig2} shows the plot of the  skewness coefficient \eqref{skew1} as a function of $\mu$ and $\alpha.$  Unlike the negative binomial, the fPd can accommodate both left-skewed and right-skewed count data making it more flexible. 
\begin{figure}[h!t!b!p!]
\centering
\includegraphics[scale=.4]{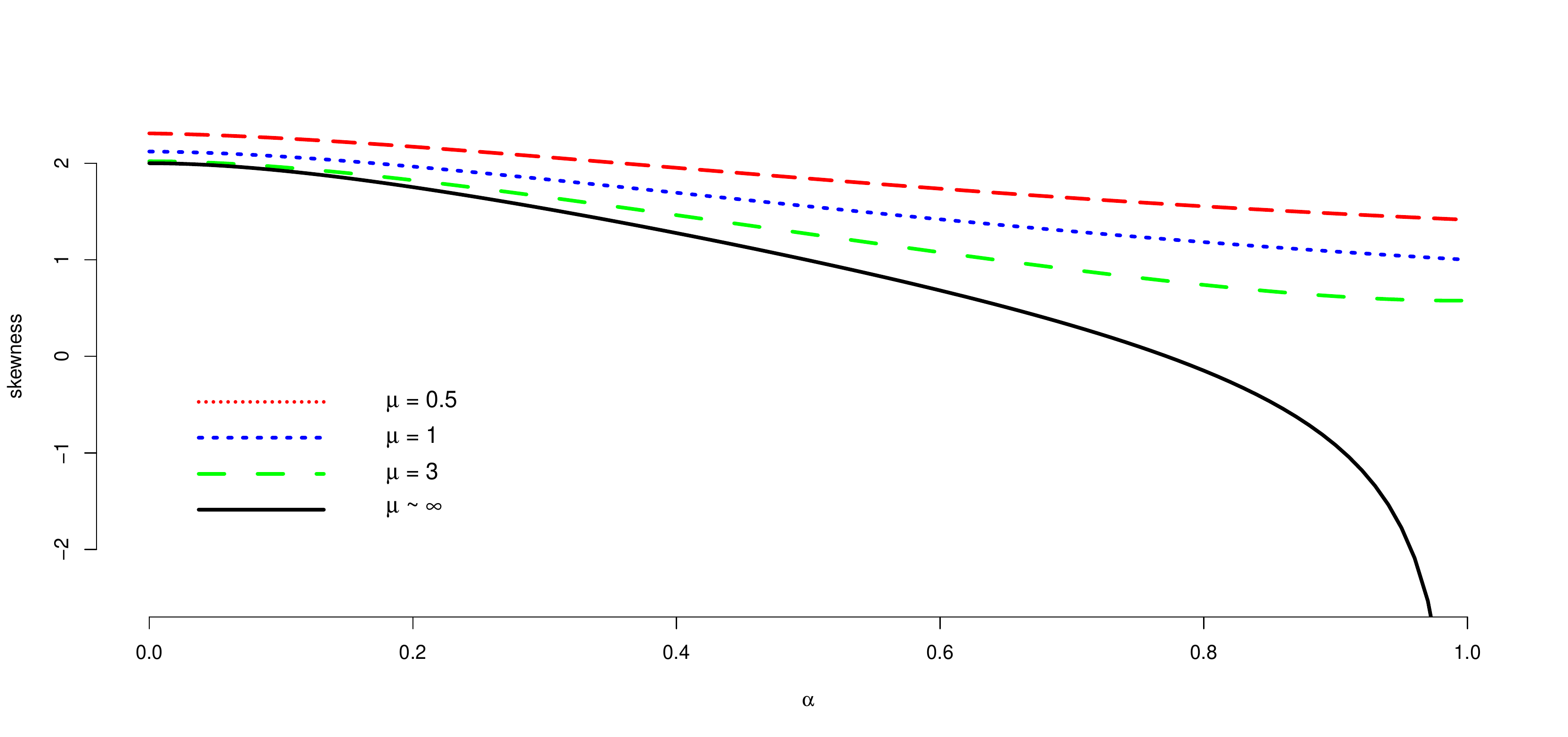}
\caption{Skewness coefficient for $\mu = 0.5,1,3$ and its limit as functions of $\alpha \in (0,1).$}
\label{Fig2}
\end{figure}
Thus, the fPd is more flexible than the negative binomial, especially if the number of failures  becomes large.

We applied the fractional Poisson model $\text{fPd}(\alpha, \mu)$  to two data sets, named
Data $1$ and Data $2,$ which are about 
the reported incidents of crime that occurred in the city of Chicago from 2001 to present\footnote{\texttt{https://data.cityofchicago.org/Public-Safety/Crimes-2001-to-present/ijzp-q8t2/data}}. The  sample distributions together with their description are shown in Figure \ref{fig3}. 
 
\begin{figure}[h!t!b!p!]
\centering
\includegraphics[scale=.45]{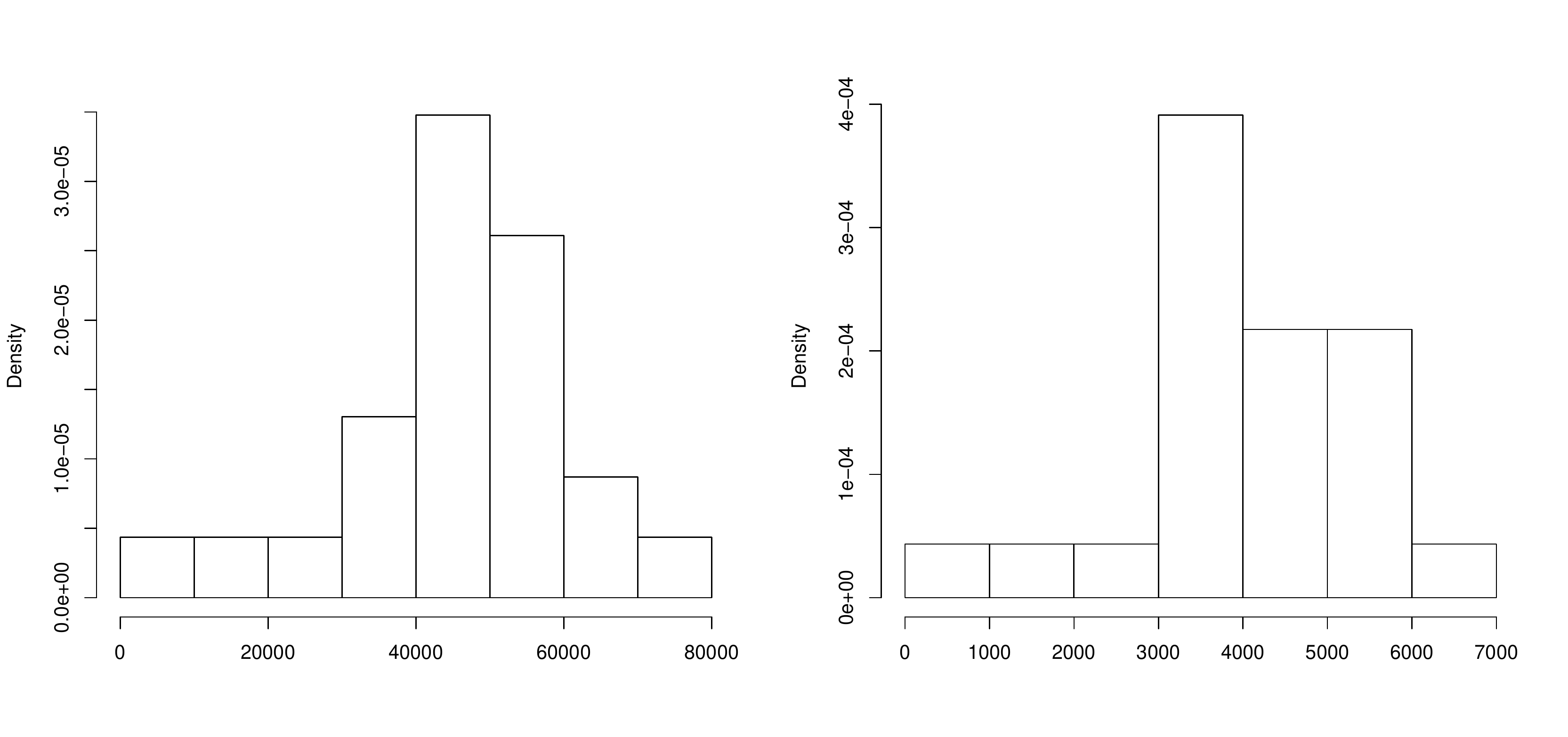}
\caption{(Left) The number of all incidents from 2001-2018 for each police district. (Right) The number of incidents described as "$\$500$ AND UNDER" for each police district.}
\label{fig3}
\end{figure}

Furthermore, we compared $\text{fPd}(\alpha, \mu)$  with the negative binomial $\text{NegBinom}(\text{size}, \text{mean})$  using the  usual chi-square goodness-of-fit test statistic and the maximum likelihood estimates for both models. Note that the chi-square test statistic follows, approximately, a chi-square distribution with $(k - 1-p)$ degrees of freedom where $k$ is the number of   cells and $p$ is the number of parameters to be estimated plus one. 

For illustration purposes, we used grid search for the $\text{fPd}(\alpha, \mu)$  as it is relatively fast due to  $\alpha$ being bounded in $(0,1)$ and to $\mu$, which is just  in the neighborhood of the true data mean scaled by $\Gamma (1 + \alpha).$ Observe that $5 \times 10^5$ random numbers are used in all the calculations.  From the results below, the fractional Poisson distribution  $\text{fPd}(\alpha, \mu)$ provides better fits than the negative binomial $\text{NegBinom}(size, mean)$  model for both data sets at $5 \%$ level of significance.  This exercise clearly demonstrates  the limitation of the negative binomial in dealing with left-skewed count data.

\begin{table}[h!t!b!p!]
\centering
\caption{\emph{Comparison between} $\text{fPd}(\alpha, \mu) $\emph{ and} $\text{NegBinom}(size, mean)$
\emph{fits}.} 

\begin{small}
\begin{tabular*}{5.6in}{l|ll}

\hline
Estimates & fPd & NegBinom  \\
\hline \\
MLE for Data 1 & $(\hat \alpha,  \hat \mu)  =(0.866, 41574.1)$  &   $ (size ,      mean) =  (1.602, 45590.17)$ \\
MLE for Data 2 & $(\hat \alpha ,\hat \mu) =  (0.85,  3607)$   &   $ (size,mean) =  (1.69, 4019.61 )$  \\ \\
Chi-square for Data 1 & 71191.64 &  202542.7 \\ 
Chi-square for Data 2 &  6442.634 & 21819.39 \\
\\
P-value for Data 1, $df= 70939$&  0.254 &  0 \\ 
P-value for Data 2, $df=6442$ &  0.495 & 0 \\ \\
\hline
 \end{tabular*}
\end{small}
  \label{t1}
\end{table}

\subsection{The case for gfPd$(\alpha, \alpha, 1, \mu)$   }

When $\beta=\alpha$ and $\delta=1$, we have   $X_{\alpha, \alpha} \; \stackrel{d}{=} \; $gfPd$(\alpha, \alpha, 1, \mu)$ with 
\begin{equation}
P(X_{\alpha, \alpha} = x) =   \Gamma (\alpha )   \mu^x E_{\alpha,\alpha (x +1) }^{x+1} (-\mu), \qquad \mu >0;\, x \in \mathbb{N}; \, \alpha \in (0, \;1].
\end{equation}

\begin{prop} The probability mass function can be written as 
\begin{equation}
P(X_{\alpha, \alpha} = x) =  \Gamma (\alpha +1) \frac{\mu^x}{x!} \int_{\mathbb{R}^+}y^{-\alpha(x+1)} e^{-\mu y^{-\alpha}} \nu_S(dy),
\end{equation}
where $\nu_S$ is the distribution of a random variable $S$ whose density has Laplace transform $\exp (-t^\alpha)$.
\end{prop}
\begin{proof} Note that  
\begin{align}
\!\!\!\!  & \frac{1}{x!}\int_{\mathbb{R}^+}y^{-\alpha(x+1)} e^{-\mu y^{-\alpha}} \nu_S(dy) =  \frac{1}{x!} \sum_{k \geq 0} \frac{(-\mu)^k}{k!}  \int_{\mathbb{R}^+}y^{-\alpha k -\alpha(x+1)} \nu_S (dy)  \\
& = \frac{1}{x!} \sum_{k \geq 0} \frac{(-\mu)^k}{k!} \frac{\Gamma ( 1 + k+x +1) }{\Gamma (1 +\alpha k + \alpha (x+1) )} \notag 
= \sum_{k \geq 0} \frac{(-\mu)^k}{k!} \frac{  \Gamma ( k+x +1) }{\alpha  x! \Gamma (\alpha k + \alpha (x+1))} \\ 
& = \frac{1}{\alpha} E_{\alpha,\alpha (x +1) }^{x+1} (-\mu). \notag
\end{align} 
\end{proof}

The above result provides an algorithm to evaluate the  probability mass function as 
\begin{align}
P(X_{\alpha, \alpha} = x) & =   \Gamma (\alpha +1  )   \frac{\mu^x}{x!} \bm{E}\left( S^{-\alpha(x+1)} e^{-\mu S^{-\alpha}}   \right) \\
& \approx \Gamma (\alpha +1  )   \frac{\mu^x}{x!} \left(\frac{1}{N} \sum_{j=1}^N S_j^{-\alpha(x+1)} e^{-\mu S_j^{-\alpha}}  \right). \notag
\end{align}
Thus, we can now estimate $\alpha$ and $\mu$ using maximum likelihood just like in the fPd case. The maximum likelihood estimates for the two crime datasets above are given in Table \ref{xxx} below. The chi-square goodness-of-fit test statistics  are   large, indicating  bad fits.   

\begin{table}[h!t!b!p!]
\centering
\caption{\label{xxx}\emph{Maximum likelihood estimates} for  $\text{gfPd}(\alpha, \alpha, 1, \mu)$.}

\begin{small}
\begin{tabular*}{3in}{l|l}
\hline
Estimates & $\text{gfPd}(\alpha, \alpha, 1, \mu)$  \\
\hline \\
MLE for Data 1 & $(\hat \alpha,  \hat \mu)  =(0.844, 38276)$    \\
MLE for Data 2 & $(\hat \alpha ,\hat \mu) =  (0.794, 3020  )$    \\ \\
Chi-square for Data 1 &  15609324   \\ 
Chi-square for Data 2 &   966402.5  \\
\\
\hline
 \end{tabular*}
\end{small}
  \label{t2}
\end{table}

\begin{Remark}
From \eqref{momf}, the random variable $X_{\alpha,\alpha}$ has factorial moments
\begin{equation}
a_{k} = \frac{\Gamma (\alpha ) \mu^{k} k!}{\Gamma ( \alpha  + \alpha k )}, \qquad k \geq 0.
\label{factfPd2}
\end{equation}
Thus the pgf is $G_{X_{\alpha,\alpha}}(u)= \Gamma (\alpha) E_{\alpha,\alpha}^1 \left( \mu \left( u-1\right)\right)$, $|u| \le 1$.
\end{Remark}

From \eqref{skew} and \eqref{factfPd2},  the symmetry structure of $X_{\alpha, \alpha}$ can be determined as follows:
\begin{equation}
\gamma(X_{\alpha, \alpha})= \frac{\Gamma (\alpha) \left(\frac{6}{ \Gamma (4 \alpha) } + \frac{6}{\mu \Gamma (3\alpha) } + \frac{1}{\mu^2 \Gamma (2\alpha) } -\frac{6}{\Gamma (2\alpha)\Gamma (3\alpha) } +\frac{2}{\Gamma (2\alpha)^3}  - \frac{3}{\mu \Gamma (2\alpha)^2} \right)}{ \left(\frac{1}{\mu \Gamma (2\alpha) } + \frac{2}{\Gamma (3\alpha) } - \frac{1}{ \Gamma (2\alpha)^2} \right)^{3/2}  }.  
\end{equation}
Moreover,
\begin{equation}
\lim_{\mu \to \infty}  \gamma(X_{\alpha,\alpha}) = \frac{\Gamma (\alpha) \left(\frac{6}{ \Gamma (4 \alpha) }   -\frac{6}{\Gamma (2\alpha)\Gamma (3\alpha) } +\frac{2}{\Gamma (2\alpha)^3}  \right)}{ \left(\frac{2}{\Gamma (3\alpha) } - \frac{1}{ \Gamma (2\alpha)^2} \right)^{3/2}  }  \neq 0, 
\label{skew2}
\end{equation}
which vanishes if $\alpha=1$ (Poisson distribution). Moreover, (\ref{skew2}) is non-negative and decreasing: this explains the  bad fits indicated by the large chi-square values above.

\section{Underdispersion and overdispersion for weighted Poisson distributions}\label{wpdd}

Weighted Poisson distributions \cite{rao65} provide a unifying approach for modelling both  overdispersion and underdispersion \cite{koko2006}. Let $Y$ be a Poisson random variable of parameter $\lambda >0$ and let $Y^{w}$ be the corresponding WPD with weight function $w$. 
\begin{Theorem}
If $\bm{E}w(Y+k)<\infty$
for all $k \in \mathbb{N}$, and $a_k=\lambda^k h(\lambda,k)$, where $h(\lambda,k)=\frac{ \bm{E}w(Y+k)}{ \bm{E}w(Y)}$, satisfies \eqref{condition}, then $Y^{w}$ has factorial moments $a_k$.
\end{Theorem}
\begin{proof}
It is enough to observe that the pgf $G_{Y^{w}}(u)$ can be written in form \eqref{forma} as follows:
\begin{align}
G_{Y^{w}}(u) & =  \sum_{k \geq 0} (u + 1 -1 )^k \frac{e^{-\lambda} \lambda^k w(k)}{k! \bm{E}w(Y)} =  \sum_{k \geq 0} \frac{(u - 1)^k}{k!} \sum_{j \geq 0} \frac{e^{-\lambda} \lambda^{j+k} w(j+k)}{j! \bm{E}w(Y)} \\
&= \sum_{k \geq 0} \frac{(u - 1)^k}{k!}  \lambda^k h(\lambda,k). \notag
\end{align} 
\end{proof}
Let $T$ be the linear left-shift operator acting on number sequences. Let us still denote with $T$ its coefficientwise extension to the ring of formal power series in ${\mathbb R}_+[[\lambda]]$ \cite{Stanley}.
Next proposition links overdispersion and underdispersion of $Y^w$ respectively to a Tur\'an-type and a reverse Tur\'an-type inequality involving $T$.
\begin{Theorem}\label{rop}
The random variable $Y^w$ is overdispersed (underdispersed) if and only if
\begin{align}
    \label{cond1}
    f(\lambda) T^2 f(\lambda) > (<)\: [T f(\lambda)]^2,
\end{align}
where $f(\lambda) = \bm{E}w(Y)$.
\end{Theorem}
\begin{proof}
The random variable $Y^w$ is overdispersed if and only if $a_2 > a_1^2$, that is 
$\bm{E}w(Y)\bm{E}w(Y+2)>[ \bm{E}w(Y+1)]^2.$ 
Equivalently,
\begin{equation}
\left(\sum_{k \geq 0}  \frac{ \lambda^k}{k!} w(k) \right)\left( \sum_{k \geq 0}  \frac{ \lambda^{k}}{k!} w(k+2) \right) > \left( \sum_{k \geq 0} \frac{ \lambda^{k}}{k!} w(k+1) \right)^2,
\label{(cond1)}
\end{equation}
and the result follows observing that 
$T^j f(\lambda) = \sum_{k \geq 0}  \frac{\lambda^k}{k!}
T^j[ w(k)]$ for $j=1,2$.
\end{proof}
\begin{Remark} Observe that when $w$ does not depend on $\lambda$, then $T^j f(\lambda) = D_{\lambda}^j f(\lambda)$ for $j=1,2.$ In this case, condition \eqref{cond1} is equivalent to $f(\lambda) D^2_{\lambda}f(\lambda)
> (<)\:[D_{\lambda}f(\lambda)]^2$, i.e.\ log-convexity (log-concavity) of $f$. This is already known in the literature (see Theorem 3 of \cite{koko2006}). 
\end{Remark}

\begin{Remark}
Note that from \eqref{(cond1)} we have 
\begin{equation}
\sum_{k \geq 0} \frac{\lambda^k}{k!} \left( \sum_{j=0}^k  \binom{k}{j} w(j) w(k-j+2) \right) > \sum_{k \geq 0}  \frac{\lambda^k}{k!} \left( \sum_{j=0}^k \binom{k}{j} w(j+1) w(k-j+1) \right)
\end{equation}
and some algebra leads us to the following sufficient condition for overdispersion or underdispersion: the random variable $Y^w$ is overdispersed (underdispersed) if
    \begin{align}
       \sum_{j=0}^{k+1} \left[\binom{k}{j}- \binom{k}{j-1} \right] w(j) w(k-j+2) > (<) \: 0 .
       \label{(32)}
    \end{align}
\end{Remark}

Notice that $\bm{E}w(Y)$ is a function of the Poisson parameter $\lambda$. For the sake of clarity, from now on, let us denote it by $\eta(\lambda)$.
Weighted Poisson distributions with a weight function $w$ not depending on the Poisson parameter $\lambda$ are also known as power series distributions (PSD) \cite{John2005} and it is easy to see that the factorial generating function in this case reads
\begin{equation}
Q(t)=\frac{\eta[\lambda(t+1)]}{\eta(\lambda)}
\end{equation}
with factorial moments 
\begin{equation} a_r=\frac{\lambda^r}{\eta(\lambda)} \frac{d^r}{d\lambda^r} [\eta(\lambda)],  \qquad r \geq 1.
\end{equation}
A special well-known family of PSD is the generalized hypergeometric probability distribution (GHPD) \cite{kemp1968},
where
\begin{equation}
Q(t) = \frac{_pF_q \left[(a);(b); \lambda(t+1) \right]}{_pF_q \left[ (a) ;(b); t \right]}
\end{equation}
with $_pF_q$ given in \eqref{hyper}.
Depending on the values of the parameters of GHPD both overdispersion and underdispersion are possible \cite{trip1979}. For $p=q=1,$ a special case of GHPD is 
the hyper-Poisson distribution \cite{CB1964}.
In the next section we will analyze an alternative WPD in which the hyper-Poisson distribution remains a special case and that exhibits both underdispersion and overdispersion.

\subsection{A novel flexible WPD allowing overdispersion or underdispersion}\label{further}

Let $Y^w$ be a WP random variable with weight function
\begin{align}
    w(k) = \frac{\Gamma(k+\gamma)}{\Gamma(\alpha k +\beta)^\nu},
\end{align}
where $\gamma > 0$, $ \min(\alpha, \beta,\nu)\ge 0$, $\alpha+\beta>0$. Moreover, if $\gamma=\beta$ and $\nu \ge 1$ then $\beta$ is allowed to be zero.
Since it is a PSD, the random variable $Y^w$ is characterized by the normalizing function
\begin{align}
    \eta(\lambda) = \eta_{\alpha,\beta}^{\gamma,\nu}(\lambda) = \sum_{k=0}^\infty \frac{\lambda^k}{k!} \frac{\Gamma(k+\gamma)}{\Gamma(\alpha k +\beta)^\nu}.
\end{align}
The convergence of the above series can be ascertained as follows. Let $\gamma \le 1$; by Gautschi's inequality (see \cite{qi}, formula (2.23)) we have the upper bound
\begin{align}
    \eta(\lambda) \le \frac{\Gamma(\gamma)}{\Gamma(\beta)^\nu} + \sum_{k=1}^{\infty} \frac{\lambda^k k^{\gamma-1}}{\Gamma(\alpha k + \beta)^\nu},
\end{align}
which converges by ratio test and taking into account the well-known asymptotics for the ratio of gamma functions (see \cite{tricomi}).
Now, let $\gamma > 1$. In this case an upper bound can be derived by formula (3.72) of \cite{qi}:
\begin{align}
    \eta(\lambda) < \frac{\Gamma(\gamma)}{\Gamma(\beta)^\nu} + \sum_{k=1}^{\infty} \frac{\lambda^k (k+\gamma)^{\gamma-1}}{\Gamma(\alpha k + \beta)^\nu}.
\end{align}
Again, this converges by ratio test and recurring to the above-mentioned asymptotic behaviour of the ratio of gamma functions.

The random variable $Y^w$ specializes to some well-known classical random variables. Specifically, we recognize the following:

\begin{enumerate}
    \item If $\gamma=\beta=\alpha=\nu=1$, we recover the Poisson distribution as the weights equal unity for each $k$.
    \item If $\gamma=\beta=\alpha=1$, we recover the COM-Poisson distribution \cite{COMPOISS} of Poisson parameter $\lambda$ and dispersion parameter $\nu$.
    \item If $\gamma=\alpha=\nu=1$ we obtain the hyper-Poisson distribution \cite{CB1964}.
    \item If $\gamma=\nu=1$ we obtain the alternative Mittag--Leffler distribution considered e.g.\ in \cite{CB1964} and \cite{Herrmann}.
    \item If $\gamma=1$ we recover the fractional COM-Poisson distribution \cite{Garra17}.
    \item If $\nu=1$ we obtain the alternative generalized Mittag-Leffler distribution \cite{tomovski}.
\end{enumerate}

Since $Y^w$ is a PSD, it is easy to derive its factorial moments,
\begin{align} \label{factmoment}
    a_r = \frac{\lambda^r}{\eta_{\alpha,\beta}^{\gamma,\nu}(\lambda)} \sum_{k=r}^\infty \frac{\lambda^{k-r}}{(k-r)!} \frac{\Gamma(k+\gamma)}{\Gamma(\alpha k+\beta)^\nu}
    = \lambda^r \frac{\eta_{\alpha,\alpha r+\beta}^{\gamma+r,\nu}(\lambda)}{\eta_{\alpha,\beta}^{\gamma,\nu}(\lambda)},
\end{align}
from which the moments are immediately derived by recalling formula \eqref{momfac}.
\begin{Remark}
\label{faa}
Since $\eta_{\alpha,\alpha r+\beta}^{\gamma+r,\nu}(\lambda) = \sum_{j \geq 0} \frac{\lambda^j}{j!} A_{j,r}$ with
\begin{equation}
A_{j,r} =  \frac{\Gamma(j + r + \gamma)}{\Gamma[\alpha (j+r)+\beta)]^{\nu}},
\end{equation}
by using Fa\`a di Bruno's formula \cite{Stanley} one has
\begin{align}
    a_r = \lambda^r \sum_{j \geq 0} \frac{\lambda^j}{j!} \sum_{i=0}^j \binom{j}{i} A_{j-i,r} D_i \quad \text{with} \quad D_i= \sum_{k=0}^i (-1)^k A_{0,0}^{-(k+1)} 
    {\mathfrak B}_{i,k}(A_{1,0}, \ldots, A_{i-k+1,0}),
\end{align}
where $\{A_{j,0}\}$ and and $\{{\mathfrak B}_{i,k}\}$ are the coefficients of $\eta_{\alpha, \beta}^{\gamma,\nu}(\lambda)$ and the partial Bell exponential polynomials \cite{Stanley}, respectively.
\end{Remark}
Furthermore, the probability mass function reads
\begin{align}
    P(Y^w=x) = \frac{\lambda^x}{x!}\frac{\Gamma(x+\gamma)}{\Gamma(\alpha x + \beta)^\nu}\frac{1}{\eta_{\alpha,\beta}^{\gamma,\nu}(\lambda)}, \qquad x \ge 0.
\end{align}

Concerning the variability of $Y^w$, by using Theorem 3 of \cite{koko2006}, the preceding Lemma and the succeeding Corollary, that is by imposing log-convexity (log-concavity) of the weight function, we write for $y \in \mathbb{R}_+$,
\begin{align}
    \frac{d^2}{d y^2} \log \frac{\Gamma(y+\gamma)}{\Gamma(\alpha y + \beta)^\nu} & = \frac{d}{dy} \left[ \frac{1}{\Gamma(y+\gamma)} \frac{d}{dy}\Gamma(y+\gamma) - \frac{\nu}{\Gamma(\alpha y + \beta)} \frac{d}{dy} \Gamma(\alpha y+\beta) \right] \\
    & = \frac{d}{dy} \left[ \psi(y+\gamma) - \nu \alpha \, \psi (\alpha y+\beta) \right], \notag
\end{align}
where $\psi(z)$ is the Psi function 
(see \cite{leb72}, Section 1.3).
In addition, by considering formula (6.4.10) of \cite{as},
\begin{align}
    \frac{d^2}{d y^2} \log \frac{\Gamma(y+\gamma)}{\Gamma(\alpha y + \beta)^\nu} = \sum_{r=0}^\infty (y+\gamma+r)^{-2} - \nu \alpha^2 \sum_{r=0}^\infty (\alpha y+\beta +r)^{-2}.
\end{align}
Therefore log-convexity (log-concavity) of $w(y)$ is equivalent to the condition
\begin{align}
    \label{group}
    \nu < (>) \, \frac{\sum_{r=0}^\infty (y+\gamma+r)^{-2}}{\alpha^2\sum_{r=0}^\infty (\alpha y+\beta +r)^{-2}}, \qquad \forall \: y \in \mathbb{R}_+.
\end{align}
This yields that if \eqref{group} holds, then $Y^w$ is overdispersed (underdispersed).

\begin{Remark}[Classical special cases]
    If $\alpha=\beta=\gamma=1$, then $Y^w$ is the COM-Poisson random variable and \eqref{group} correctly reduces to the ranges $\nu>1$ giving underdispersion and $\nu \in [0,1)$ giving overdispersion.
    If $\alpha=\gamma=\nu=1$, then $Y^w$ is the hyper-Poisson random variable and \eqref{group} correctly reduces to the ranges $\beta>1$ (overdispersion) and $\beta \in [0,1)$ (underdispersion).
    This holds as $\beta \mapsto \sum_{r=0}^\infty (y+\beta+r)^{-2}$ is decreasing for all fixed $y \in \mathbb{R}_+$.
\end{Remark}

In the two next sections we analyze  two special cases of interest, the first of which, to the best of our knowledge, is still not considered in the literature.

\subsubsection{Model I}\label{I}

We first introduce the special case in which $\alpha=1$, $\gamma=\beta$, $\beta>0$, and $\beta$ is allowed to be zero only if $\nu \ge 1$. This is a three-parameter  ($\lambda,\nu,\beta$) model which retains the same simple conditions for underdispersion and overdispersion as for the COM-Poisson model. Indeed, formula \eqref{group} reduces to $\nu>1$ and $\nu \in [0,1)$, respectively. However, this model is more flexible than the COM-Poisson model because of the presence of the parameter $\beta$.  Notice that the pmf can be written as 
\begin{equation}\label{tre}
P(Y^w = x) = \frac{1}{x!}\exp \left(x  \log \lambda + (1-\nu) \log  \Gamma (x + \beta) -\log  \eta_{1,\beta}^{\beta,\nu}(\lambda)  \right),    
\end{equation}
which suggests that Model I belongs to the exponential family of distributions with parameters $\log \lambda$ and $1-\nu$, where $\beta$ is a nuisance parameter or is known. Figures \ref{Fig4} and \ref{Fig5} show sample shapes of this family of distributions.

\begin{figure}[h!t!b!p!]
\centering
\includegraphics[scale=.45]{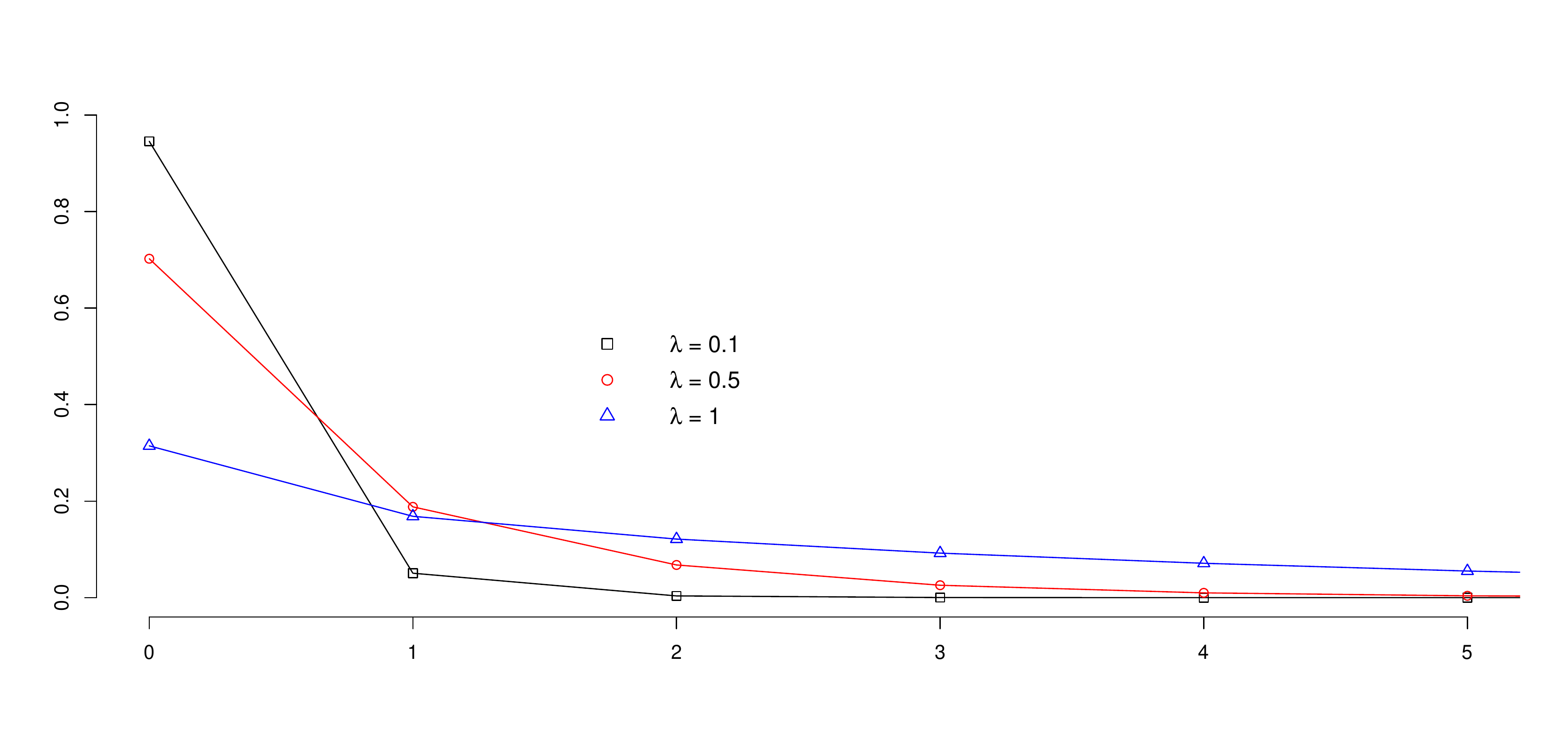}
\caption{Probability mass functions \eqref{tre} for $\lambda =0.1, 0.5, 1$,  $\beta = 0.5$, and $\nu=0.1.$}
\label{Fig4}
\end{figure}

\begin{figure}[h!t!b!p!]
\centering
\includegraphics[scale=.45]{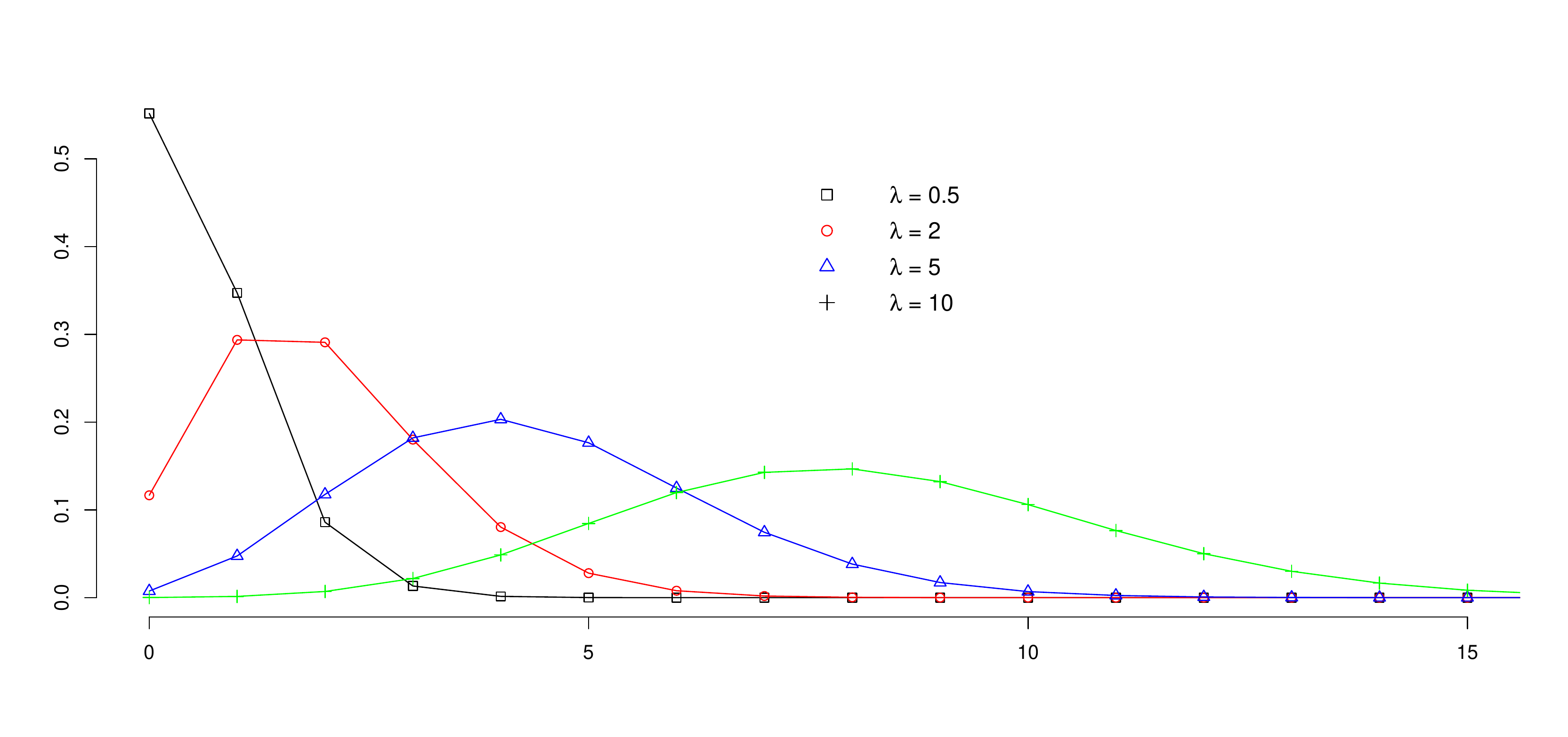}
\caption{Probability mass functions \eqref{tre} for $\lambda =0.5, 2, 5, 10$,  $\beta = 0.1$, and $\nu= 1.1.$}
\label{Fig5}
\end{figure}
Note that distributions in Figure \ref{Fig4} (Figure \ref{Fig5}) are   overdispersed (underdispersed). Also, 
\begin{equation}
    P(Y^w = x+1) = \frac{\lambda}{(x+1)(x+\beta)^{\nu-1}} P(Y^w = x).
\end{equation}
This gives a procedure to calculate iteratively the probability mass function and generate random numbers. The only thing to figure out is to compute $\eta_{1,\beta}^{\beta,\nu}(\lambda)$ in order to obtain $P(Y^w = 0)=1/[\Gamma (\beta)^{\nu-1} \eta_{1,\beta}^{\beta,\nu}(\lambda)]$.

An upper bound for the normalizing function $\eta_{1,\beta}^{\beta,\nu}(\lambda)$ can be determined similarly to \cite{mink}, Section 3.2, taking into consideration that the multiplier
\begin{align}
    \label{multi}
    \lambda (j+\beta)^{1-\nu}/(j+1)
\end{align}
is ultimately monotonically decreasing. Hence, we can approximate the normalizing constant  $\eta_{1,\beta}^{\beta,\nu}(\lambda)$ by truncating the series and bound the truncation error $R_{\widetilde{k}}$,
\begin{align}
    \eta_{1,\beta}^{\beta,\nu}(\lambda) &= \sum_{j=0}^{\widetilde{k}} \frac{\lambda^j}{j!} \Gamma(j+\beta)^{1-\nu} + R_{\widetilde{k}} \\
    &< \sum_{j=0}^{\widetilde{k}} \frac{\lambda^j}{j!} \Gamma(j+\beta)^{1-\nu} + \frac{\lambda^{\widetilde{k}+1}\Gamma(\widetilde{k}+1+\beta)^{1-\nu}}{(\widetilde{k}+1)!} \sum_{j=0}^\infty \varepsilon_{\widetilde{k}}^j \notag \\
    &< \sum_{j=0}^{\widetilde{k}} \frac{\lambda^j}{j!} \Gamma(j+\beta)^{1-\nu} + \frac{\lambda^{\widetilde{k}+1}\Gamma(\widetilde{k}+1+\beta)^{1-\nu}}{(\widetilde{k}+1)!\,(1-\varepsilon_{\widetilde{k}})}, \notag
\end{align}
where $\widetilde{k}$ is such that for $j>\widetilde{k}$ the multiplier \eqref{multi} is already monotonically decreasing and bounded above by $\varepsilon_{\widetilde{k}} \in (0,1)$.
Correspondingly, denoting with $\widetilde{\eta}_{1,\beta}^{\beta,\nu}(\lambda) = \sum_{j=0}^{\widetilde{k}} \frac{\lambda^j}{j!} \Gamma(j+\beta)^{1-\nu}$, the relative truncation error $R_{\widetilde{k}}/\widetilde{\eta}_{1,\beta}^{\beta,\nu}(\lambda)$ is bounded by
\begin{align}
    \frac{\lambda^{\widetilde{k}+1}\Gamma(\widetilde{k}+1+\beta)^{1-\nu}}{(\widetilde{k}+1)!\,(1-\varepsilon_{\widetilde{k}})\widetilde{\eta}_{1,\beta}^{\beta,\nu}(\lambda)}.
\end{align}

As a last remark, we can further simplify the model obtaining a two-parameter model. In order to do so, let $\nu=\beta$, with $\beta>0$. The obtained model still allows for underdispersion ($\beta >1$) and overdispersion ($\beta \in (0,1)$) and it should be directly compared with the COM-Poisson and the hyper-Poisson models.

\subsubsection{Model II}
If we set $\alpha=\nu=1$ we get  another three-parameter ($\lambda,\gamma,\beta$) model, special case of the alternative generalized Mittag-Leffler distribution (see point 6 above). The reparametrization $\beta=\xi \gamma$ together with condition \eqref{group} shows that both overdispersion ($\xi>1$) and underdispersion ($\xi \in (0,1)$) are possible. This comes from the fact that $\omega \mapsto \sum_{r=0}^\infty (y+\omega+r)^{-2}$ is decreasing for all fixed $y \in \mathbb{R}_+$. As for Model I,
the probability  distribution belongs to the exponential family with parameter $\log \lambda,$ with $\gamma$ and $\beta$ as nuisance parameters.
Explicitly, the pmf reads
\begin{align}\label{tru}
    P(Y^w=x) = \frac{\lambda^x}{x!}\frac{\Gamma(x+\gamma)}{\Gamma(x + \beta)}\frac{1}{\eta_{1,\beta}^{\gamma,1}(\lambda)}, \qquad x \ge 0,
\end{align}
and, as in the previous Section \ref{I}, the iterative representation
\begin{equation}
     P(Y^w = x+1) = \frac{\lambda (x+\gamma)}{(x+1)(x+\beta)} P(Y^w = x),
\end{equation}
allows an approximated evaluation of the pmf with error control, and consequently random number generation. Also in this case this holds as the involved multiplier is ultimately monotonically decreasing. Figures \ref{Fig6} and \ref{Fig7} show some forms of this class of distributions. Observe that distributions in Figure \ref{Fig6} (Figure \ref{Fig7}) are   underdispersed (overdispersed).

\begin{figure}[h!t!b!p!]
\centering
\includegraphics[scale=.45]{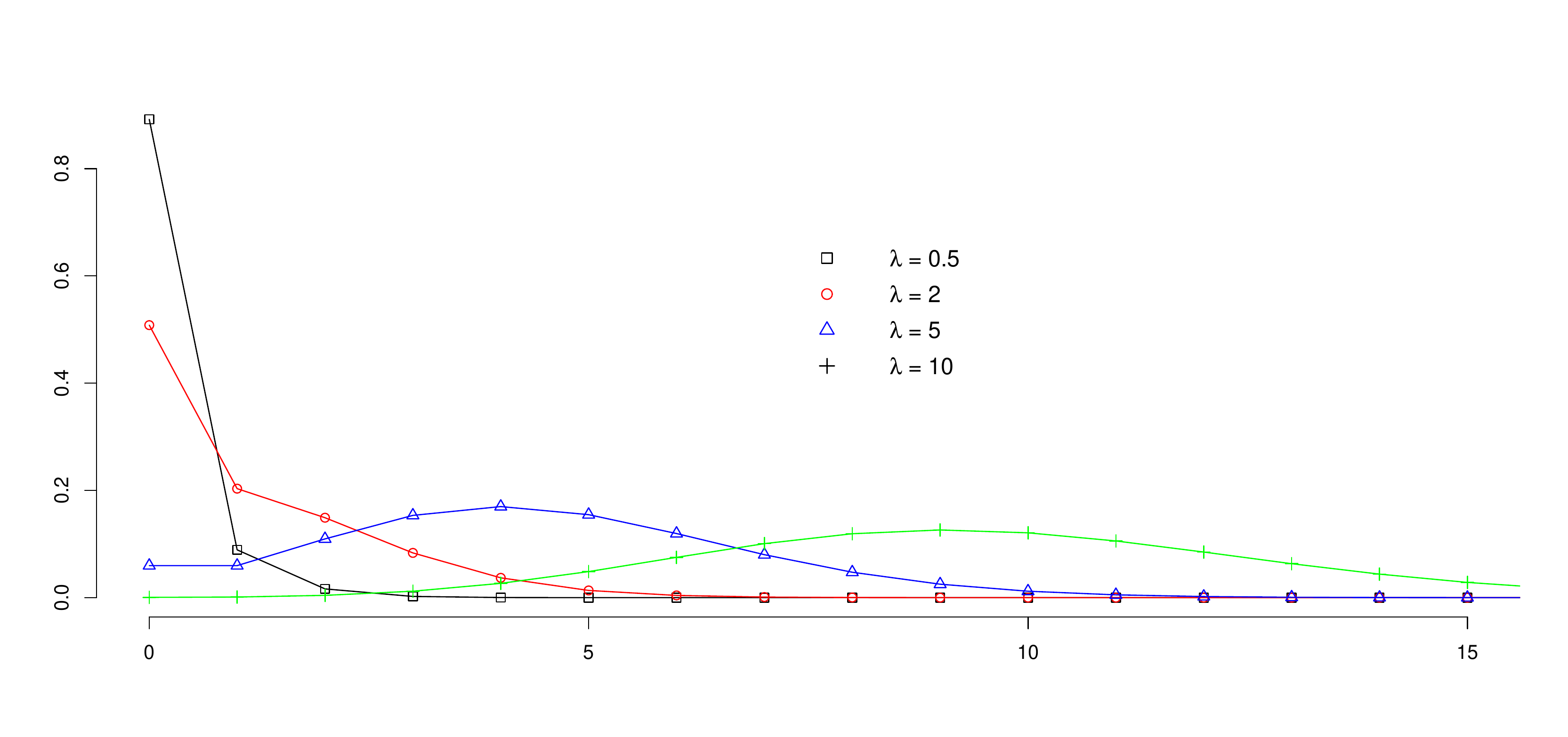}
\caption{Probability mass functions \eqref{tru} for $\lambda =0.5, 2, 5, 10$,  $\beta = 0.5$, and $\gamma=0.1.$}
\label{Fig6}
\end{figure}

\begin{figure}[h!t!b!p!]
\centering
\includegraphics[scale=.45]{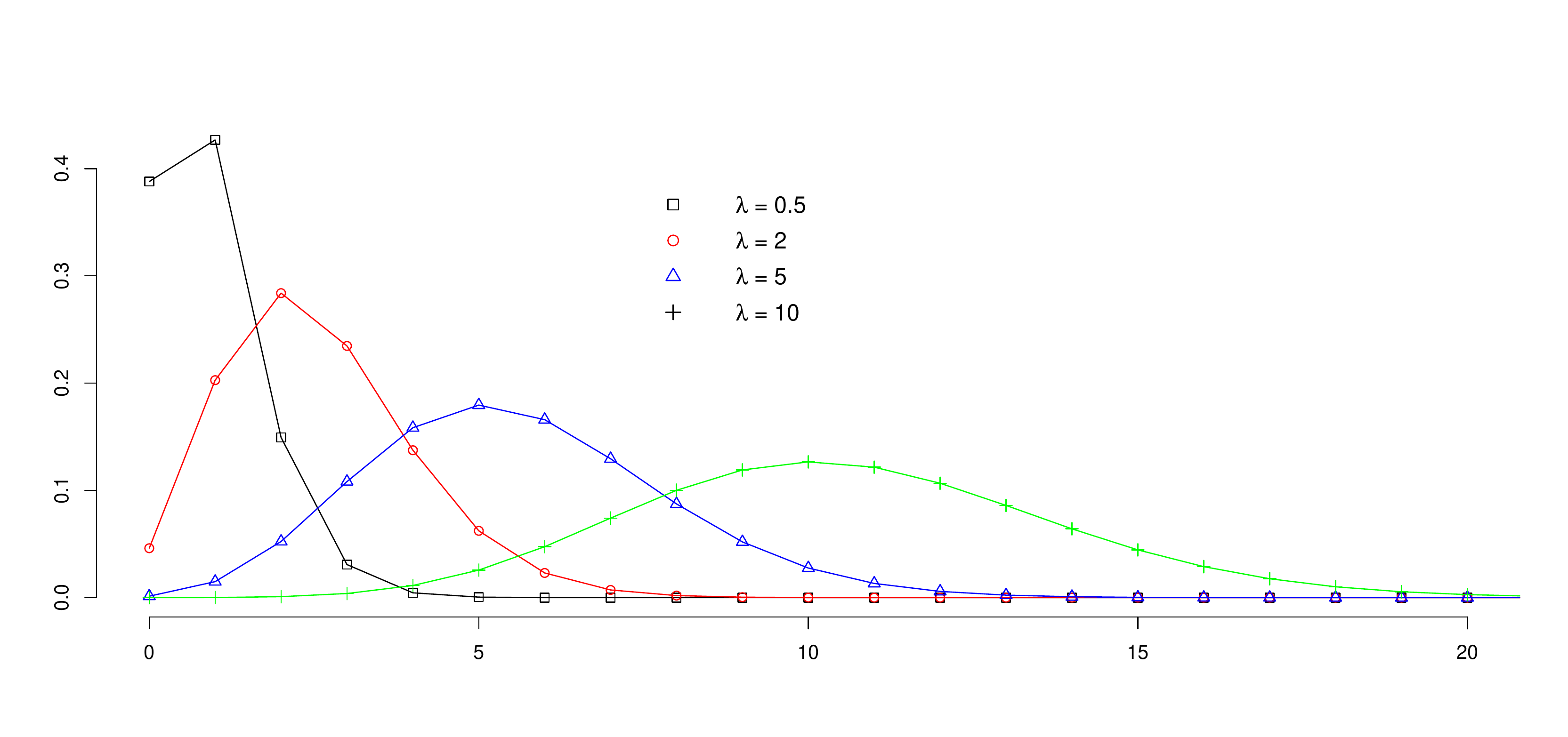}
\caption{Probability mass functions \eqref{tru} for $\lambda =0.5, 2, 5, 10$,  $\beta = 0.1$, and $\gamma= 1.1.$}
\label{Fig7}
\end{figure}

If we further let $\lambda=1$, we obtain a two-parameter model, still allowing for underdispersion if $\beta \in (0,\gamma)$ (or equivalently $\xi \in (0,1)$) and overdispersion if $\beta > \gamma$ (or $\xi > 1$), which is also directly comparable with the two-parameter Model I above, the COM-Poisson model, and the hyper-Poisson model.

\subsubsection{Comparison}

We now compare Model I and Model II with  known models that  allow overdispersion and underdispersion such as the COM-Poisson, generalized Poisson and hyper-Poisson models as cited above. Note that the hyper-Poisson distribution satisfies 

\begin{equation}
      P(Y^w = x+1) = \frac{\lambda}{(x+\beta)} P(Y^w = x).
\end{equation}

For comparison purposes, we first consider the number of fish caught  data\footnote[2]{\texttt{https://stats.idre.ucla.edu/stat/data/fish.csv}} shown in Figure \ref{fig4b} (left panel) below. The dataset corresponds to 239 groups (as 11 potential  outliers were removed)  that went to a state park  and state wildlife biologists  asked  visitors how  many fish they  caught.    The mean fish caught  is around 1.48 while the variance is 8.04.  Furthermore,  the \texttt{optimx} (for hyper-Poisson, Model I and Model II), \texttt{COMPoissonReg} (for COM-Poisson), \texttt{compoisson} (for COM-Poisson), and  \texttt{VGAM} (for generalized Poisson) packages in R are used for the maximum likelihood estimation and the chi-square goodness-of-fit tests.  In particular,  the  \texttt{L-BFGS-B} method from the \texttt{optimx} package  is  used and 1000 terms were summed for the normalizing constant $\eta_{\alpha,\beta}^{\gamma, \nu} (\lambda)$.    Just like the comparisons above, a chi-square distribution  is used as reference  where the degrees of freedom is the number of   cells minus the  number of model parameters. From Table \ref{t3}, Model I and Model II clearly outperform the other models although the generalized Poisson  and hyper-Poisson (subcase of WPD) also provide good fits to the fish count data. 

\begin{figure}[h!t!b!p!]
\centering
\includegraphics[scale=.8]{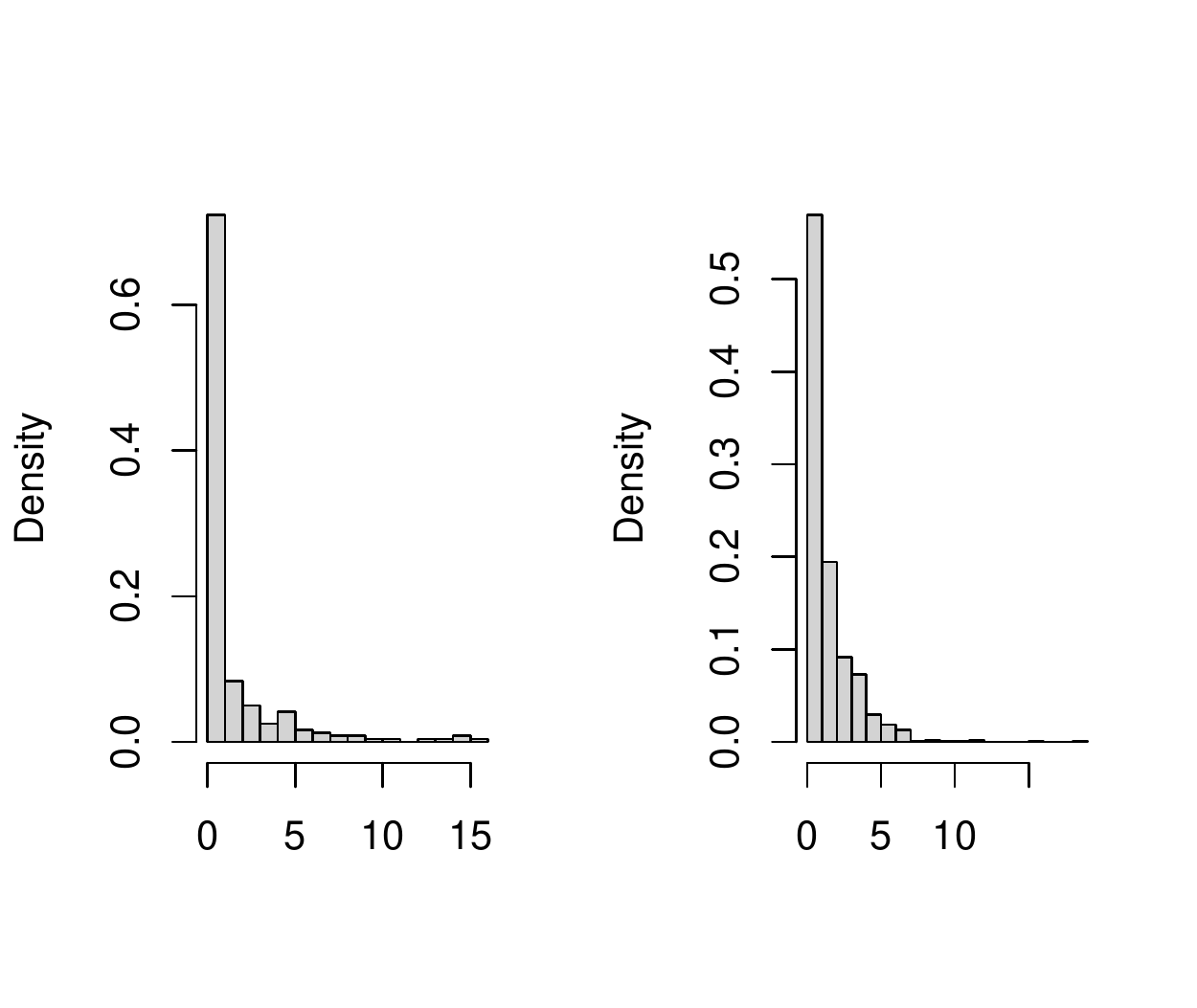}
\caption{(Left) The fish caught count data. (Right)  The count of articles produced  by  graduate students in biochemistry Ph.D. programs.}
\label{fig4b}
\end{figure}

\begin{table}[h!t!b!p!]
\centering
\caption{\emph{Comparison results for the fish count data.} } 
\begin{small}
\begin{tabular*}{5in}{l|lll}
\hline
Model & ML Estimates & Chi-square   & P-value \\
\hline  \\
COM-Poisson & $(\hat \lambda,  \hat \nu )  = (11.5876, 0.9806)$  &   $ 13495000$ & 0 \\ 
Hyper-Poisson & $(\hat \alpha ,\hat \beta) =   (37.1126, 170)$   &   $ 18.2259$ &0.1090  \\
Gen Poisson  & $(\hat \lambda ,\hat \theta) =  (0.6334,  0.5430)$   &   $ 15.1197$ & 0.2349   \\
Model I &$(\hat \alpha ,\hat \beta, \hat \nu) =  (0.9544,  0.2126, 0.0632)$   &   $ 12.6350$   & 0.3178 \\
Model II & $(\hat \alpha ,\hat \beta, \hat \gamma) =  (134.5545, 149.9958, 0.2504)$   &   $\bm{9.7697}$  & $\bm{0.5512}$ \\ \\
\hline
 \end{tabular*}
\end{small}
  \label{t3}
\end{table}

 We have also considered the \texttt{bioChemists} data from the \texttt{pscl} package in R, particularly the count of articles produced  by 915 graduate students in biochemistry Ph.D.\ programs during last 3 years in the program. The data has mean 1.69 and variance of 3.71, and  is showcased in Figure \ref{fig4b} (right panel).   Apparently, Table \ref{t4} suggests that Model II outperforms the rest of the models considered for the article count data. Overall, there is potential in WPD's (e.g.,  Model I and Model II) in flexibly capturing overdispersed and/or underdispersed count data distributions. 

\begin{table}[h!t!b!p!]
\centering
\caption{\emph{Comparison results for the article count data.} } 
\begin{small}
\begin{tabular*}{5in}{l|lll}
\hline
Model & ML Estimates & Chi-square   & P-value \\
\hline  \\
COM-Poisson & $(\hat \lambda,  \hat \nu )  = (14.4428, 0.9903)$  &   $ 14156763$ & 0 \\
Hyper-Poisson & $(\hat \alpha ,\hat \beta) =   (14.5253, 20.3124)$   &   $ 1549.086$ & 1.3487e-34   \\
Gen Poisson  & $(\hat \lambda ,\hat \theta) =  (0.2991,  1.1886)$   &   $ 121.9043$ & 8.0757e-19   \\
Model I &$(\hat \alpha ,\hat \beta, \hat \nu) =  (0.4992,  1.7028, 0.001)$   &   $ 266.8644$   & 9.229e-49 \\
Model II & $(\hat \alpha ,\hat \beta, \hat \gamma) =  (73.17587, 150.001, 1.7985)$   &   $\bm{21.4124}$  & $\bm{0.0915}$ \\ \\
\hline
 \end{tabular*}
\end{small}
  \label{t4}
\end{table}


\subsection*{Acknowledgments}

F.~Polito has been partially supported by the project ``Memory in Evolving Graphs'' (Compagnia di San Paolo/Università degli Studi di Torino).

\end{document}